\documentclass{article}

\newcommand {\theoremstyle} [1] { }
\newenvironment{proof}{{\noindent\it\underline{Proof}}:}{\hfill$\Box$}
\newenvironment{proof1}
{\noindent\it\underline{Proof of Theorem \ref{bound}}:\rm}{\hfill$\Box$}
\newtheorem{thm}{Theorem}[section]
 
 \theoremstyle{plain}
  
 \newtheorem{cor}[thm]{Corollary}
 \theoremstyle{definition}
 \newtheorem{defn}[thm]{Definition}
 
 \theoremstyle{remark}
 \newtheorem{rem}[thm]{Remark}

\usepackage {amssymb}

\usepackage {amssymb}
\usepackage{amsmath,amsfonts, color}
\usepackage{graphicx}

\title{{ On the Solvability of the Periodically Forced Relativistic Pendulum Equation on Time Scales}}
\author{{\large P. Amster, M. P. Kuna, D. P. Santos}
{\bf\large}\vspace{1mm}\\
{\small Departamento de Matem\'atica, Facultad de Ciencias Exactas y Naturales}\\ 
{\small Universidad de Buenos Aires \& IMAS-CONICET}\\
{\small Ciudad Universitaria. Pabellon I,(1428), Buenos Aires, Argentina.}\\
{\small pamster@dm.uba.ar -- mpkuna@dm.uba.ar -- dsantos@dm.uba.ar}
\vspace{3mm}
}

\date{}

\begin{document}

\maketitle

\begin{abstract}
We study some properties of the range of the relativistic pendulum operator $\mathcal P$, that is, the set of possible continuous $T$-periodic forcing terms $p$ for which the equation $\mathcal P x=p$ admits a $T$-periodic solution over a $T$-periodic time scale $\mathbb T$. 
Writing $p(t)=p_0(t)+\overline p$, we prove the existence of a nonempty compact interval $\mathcal I(p_0)$, depending continuously on $p_0$,  such that the problem has a solution if and only if $\overline p\in \mathcal I(p_0)$ and at least two different solutions when $\overline p$ is an interior point.  
Furthermore, we give sufficient conditions for nondegeneracy; specifically, 
we prove that if $T$ is small then $\mathcal I(p_0)$ is a neighbourhood of $0$ for 
arbitrary $p_0$. 
The results   
in the present paper improve the 
smallness condition obtained in previous works for the continuous case $\mathbb T=\mathbb R$. 

\end{abstract}
 
 \medskip

\smallskip 



 \medskip
 \noindent


\section{Introduction}
The $T$-periodic problem for the forced relativistic pendulum equation on time scales reads
\begin{equation}\label{eq1}
\mathcal Px(t):=(\varphi(x^{\Delta}(t)))^{\Delta} +ax^{\Delta}(t)+ b\sin x(t) =p_0(t) + s,  \ \ \ t\in \mathbb{T},
\end{equation}
where $a,b>0$ and $s$  are real numbers, ${\mathbb{T}}$ is an arbitrary $T$-periodic nonempty closed subset of  $\mathbb{R}$ for some  $T>0$, $\varphi:(-c,c)\rightarrow \mathbb{R}$ is the  relativistic operator  $\varphi(x):=\displaystyle \frac{x}{\sqrt{1-\frac{x^{2}}{c^{2}}}} $ with $c>0$ and $p_{0}$ is continuous and $T$-periodic in $\mathbb{T}$,  with zero average. In this work, we are concerned with the set of all possible values of $s$ such that (\ref{eq1}) admits a $T$-periodic solution.

The time scales theory was introduced in 1988, in the PhD thesis of  Stefan Hilger \cite{hilgerthe}, as an attempt to unify discrete and continuous calculus.
 The time scale $\mathbb R$ corresponds to the continuous case and, 
 hence, yields
 results for ordinary differential equations. If the time scale is ${\mathbb{Z}}$, 
then the results apply to standard difference equations. However, the generality of the set $\mathbb T$ produces many different situations in which the time scales formalism is useful  in several applications. For example, in the study of hybrid discrete-continuous dynamical systems, see \cite{bo1}.

In the past decades, periodic problems involving the relativistic forced pendulum   differential equation for the continuous case $\mathbb T=\mathbb R$ were studied by many authors, see \cite{bjm,ma1,bre,man11,torres,torres2}. In particular, the works
 \cite{bjm, torres2}  are concerned with  the so-called \textit{solvability set}, that is, the  
 set $\mathcal I(p_{0})$ of  values of $s$ for which (\ref{eq1}) has at least one $T$-periodic solution. 
 We remark 
that problem (\ref{eq1}) is $2\pi$-periodic and, consequently,  
if $x$ is a $T$-periodic solution then 
$x+2k\pi$ is also a $T$-periodic solution for all $k\in \mathbb Z$. For this reason, the multiplicity
results for (\ref{eq1}) usually refer to the existence of \textit{geometrically distinct} $T$-periodic solutions, i.e. solutions not differing by a multiple of $2\pi$. 

{For the standard pendulum equation with $a=0$, 
the solvability set was analyzed in 
the pioneering work \cite{castro}, where it was proved that 
$\mathcal I(p_{0})\subset [-b,b]$ is a nonempty compact interval containing $0$. 
Moreover, $\mathcal I(p_0)$ depends continuously on $p_0$. 
These results were partially extended to the relativistic case in \cite{bre}; however, the method of proof in both
works is variational and, consequently, cannot be applied to the case $a>0$. This latter situation was studied in \cite{fm} for the standard pendulum and in \cite{torres2} for the relativistic case.
An interesting question, stated already in \cite{castro} is whether or not the equation may be {degenerate}, namely: is there any $p_0$ such that $\mathcal I(p_0)$ reduces to a single point? 
Many works are devoted to this problem and, for the classical pendulum, 
nondegeneracy has been proved for an open and dense subset of $\tilde C_T$, the space of zero-average $T$-periodic continuous functions. 
However, the question for arbitrary $p_0$ remains unsolved. 
For a survey on the pendulum equation and open problems see for example \cite{maw-75}.  
}


The purpose of this  work is  to  
{extend   the  results  in \cite{bjm} and \cite{torres2} to the context of time scales.} 
{To this end, we prove in the first place that the set 
$\mathcal I(p_{0})$ is a nonempty compact interval depending continuously on $p_0$. }
The method of proof is inspired in a simple idea introduced in  \cite{fm} for the standard pendulum equation, which basically employs the Schauder Theorem and the method of upper and lower solutions. 
Moreover, by a Leray-Schauder degree argument it shall be proved that if $s$ is an interior point of $\mathcal I(p_0)$, then 
the problem admits at {least} two geometrically distinct periodic solutions.
  
Furthermore, sufficient   conditions 
shall be given in order to guarantee that $0\in \mathcal I(p_0)$. We recall that, when $a\ne 0$, this is not trivial even in the continuous case $\mathbb T=\mathbb R$. For the classical pendulum equation, there exist well known examples 
with $0\notin \mathcal I(p_0)$ for arbitrary values of $T$; for the relativistic case, 
it was proved in 
\cite{bjm}  that, if 
$cT < \sqrt{3} \pi$, then 
$0\in \mathcal I(p_0)^\circ$. 
It is worth noticing that, however,
 the problem is still open for large values of $T$. 
As we shall see, a slight improvement of the previous bound can be deduced from the results in the present paper. Specifically, we shall prove the existence of $T^*$ with $cT^*>\pi$ such that if $T\le T^*$ then 
$0\in \mathcal I(p_0)$   and it is an interior point when the inequality is strict.  
An inferior 
bound for $T^*$ can be characterized as a zero of a real function; for the continuous case $\mathbb T=\mathbb R$, it is shown that $cT^* >\sqrt 3\pi$ and  verified numerically that
$cT^* > 6.318$. 
We remark that the computation is independent of $p_0$: in other words, if $T<T^*$, then the range of the operator $\mathcal P$ contains a set of the form $\tilde C_T+[-\varepsilon,\varepsilon]$ for some $\varepsilon>0$. 

We highlight that our paper is devoted to equations on time scales that involve a $\varphi$-laplacian of relativistic  type, for which the literature is scarce. 
For example, in \cite{pras}, the existence of heteroclinic solutions for a family of  equations on time scales that includes the 
unforced relativistic pendulum is proved. 
However, to our knowledge there are no  papers concerned with periodic solutions and, more precisely,  the solvability set for equations with a singular $\varphi$-laplacian on time scales.

This work is organized as follows.  In Section 2,  we establish  the notation,  terminology  and   preliminary  results  which  will  be  used  throughout  the paper. In Section 3 we prove that the set $\mathcal I(p_0)$ is a nonempty compact interval depending continuously on $p_0$, and 
that two geometrically distinct $T$-periodic solutions exist when $s$ is an interior point.
Finally, Section 4 is devoted to find sufficient conditions in order to guarantee that  
$0\in \mathcal I(p_0)$ and  improve the condition obtained in \cite{bjm} for the continuous case.

\section{Notation and preliminaries}
\label{S:-1}

Fix $T>0$ and assume that $\mathbb T$ is $T$-periodic, i.e. $\mathbb T+T=\mathbb T$. Let 
$C_{T}=C_{T}\left(\mathbb{T},\mathbb{R} \right)$ be the Banach space of all continuous $T$-periodic functions  on $\mathbb{T}$  endowed with the uniform norm 
$$\left\| x \right\|_{\infty}=\displaystyle  \sup _{\mathbb{T}} |x(t)| = \sup _{\left[0,T\right]_{\mathbb{T}}} |x(t)|$$ and let $\tilde{C_T}$ be
the subspace of those elements of $C_T$ having zero average. 
By
$C^{1}_{T}=C^{1}_{T}\left(\mathbb{T},\mathbb{R} \right)$ we shall denote the Banach space  of all continuous $T$-periodic functions on $\mathbb{T}$ that are $\Delta$-differentiable functions with continuous  $\Delta$-derivatives, endowed with the usual norm
 $$\left\|x\right\|_{1}= \sup _{\left[0,T\right]_{\mathbb{T}}} |x(t)|+ \displaystyle \sup _{\left[0,T\right]_{\mathbb{T}}}|x^{\Delta}(t)|.$$

Equation (\ref{eq1}) can be written as
\begin{equation}\label{eq2}
(\varphi(x^{\Delta}(t)))^{\Delta} = f(t,x(t),x^{\Delta}(t))  \ \ \ t\in \mathbb{T},
\end{equation}
where  $f: \mathbb{T}\times \mathbb{R} \times  \mathbb{R}  \rightarrow \mathbb{R}$ is the  continuous function given by
$f(t,u,v):=p_0(t) + s - au - b\sin(u) $. 
A function $x\in C^{1}_{T}$ is said to be a solution of (\ref{eq2}) if $\varphi(x^{\Delta})\in  C^{1}_{T}$ and  verifies
$(\varphi(x^{\Delta}(t)))^{\Delta} = f(t,x(t),x^{\Delta}(t))$ for all $t\in \mathbb{T}$. We remark that necessarily $\|x\|_\infty<c$. 

For $x\in C_T$, the 
average, the  maximum value and the minimum value of $x$ shall be denoted respectively by $\overline x$, $x_{\max}$ and $x_{\min}$, namely
$$
\overline x:= \frac 1T\int_0^T x(t)\Delta t, \qquad x_{\max}:= \max_{t\in [0,T]_\mathbb T} x(t)\qquad x_{\min}:= \min_{t\in [0,T]_\mathbb T} x(t).
$$
For details on time scales theory we refer the reader to \cite{bo1,bo2}. 




\subsection{Upper and lower solutions and degree}


 Let us  define $T$-periodic lower and upper solutions for problem  (\ref{eq2}) as follows.
 
\begin{defn} A lower $T$-periodic solution $\alpha$ (resp. upper solution $\beta$) of (\ref{eq2}) is a function $\alpha \in C^{1}_T$ with  $\left\| \alpha^{\Delta} \right\|_{\infty}<c$ such that $\varphi(\alpha^{\Delta})$ is continuously $\Delta$-differentiable 
and 
\begin{equation}\label{eq3}
\left(\varphi\left(\alpha^{\Delta}(t)\right)\right)^{\Delta} \geq f(t,\alpha(t), \alpha^{\Delta}(t)) \ \ \   (resp. \  \left(\varphi\left(\beta^{\Delta}(t)\right)\right)^{\Delta} \leq f(t,\beta(t),\beta^{\Delta}(t)))
\end{equation}
for all $t\in \mathbb T$.  
Such lower (upper) solution is called 
{\rm strict} if the inequality  (\ref{eq3}) is strict for all $t\in\mathbb T$. 
\end{defn}


  



It is worth recalling   the problem of finding $T$-periodic solutions of (\ref{eq2}) over the closure of the set
$$\Omega_{\alpha,\beta} :=\{ x\in C^1_T: \alpha(t)\le x(t)\le \beta(t)\,\hbox{ for all $t$}\} 
$$
can be reduced to a fixed point equation $x=M_f(x)$, where $M_f:\overline\Omega_{\alpha,\beta}\to C^1_T$ is a compact operator that can be defined according to the nonlinear version of the continuation method (see e.g. \cite{man12}), namely
$$M_f(x):= \overline x + \overline {N_fx} + 
K(N_fx - \overline {N_fx}),
$$
where $N_f$ is the Nemitskii operator associated to $f$ and $K:\tilde C_T\to \tilde C_T$ is the (nonlinear) compact operator given by $K\xi=x$, with $x\in C_T^1$ the unique solution of the problem  
$(\varphi(x^{\Delta}(t)))^{\Delta} =\xi(t)$ 
with zero average. We recall, for the reader's convenience, that the definition of $K$ based upon
the existence, easy to prove, of a (unique) completely continuous map $c:C_T\to \mathbb R$ satisfying
$\int_0^T \varphi^{-1}(h+c(h))\Delta t=0$ for all $h\in {C_T}$. 
For the purposes of the present paper, we shall only need the following result, which is an  adaptation of Theorem {3.7} in \cite{opuscula}: 


\begin{thm}\label{te2}
 Suppose that (\ref{eq2}) has a $T$-periodic
 lower solution $\alpha$ and  an   
 upper solution $\beta$  such that $\alpha(t)\leq \beta(t)$ for all $t\in \mathbb T$. 
Then problem (\ref{eq1}) has at least one $T$-periodic solution $x$ with $\alpha(t)\le x(t)\le \beta(t)$ for all $t\in \mathbb T$. If furthermore $\alpha$ and $\beta$ are strict, then  $\deg_{LS}(I-M_f,\Omega_{\alpha, \beta}(0),0)= 1$,
where $\deg_{LS}$ stands for the Leray-Schauder degree.
\end{thm}

 \section{The solvability set $\mathcal I(p_0)$}

In this section, we shall prove that the solution set $\mathcal I(p_0)$ is a nonempty  compact set; furthermore, 
employing the method of upper and lower solutions it shall be verified that $\mathcal I(p_0)$ is an interval depending continuously on $p_0$. Finally,  the excision property of the degree will 
be employed to verify that if $s$ is an interior point of $\mathcal I(p_0)$, then the problem has at least $2$ geometrically different $T$-periodic solutions. 

\begin{thm}\label{te3}
Assume that  $p_{0}\in C_T$ has zero average. Then, there exist numbers $d(p_{0})$
and $D(p_{0})$, with   $-b\leq d(p_{0})\leq D(p_{0})\leq b$, such that (\ref{eq1}) has at least one $T$-periodic solution if and only if $s\in \left[d(p_{0}), D(p_{0}) \right]$. 
Moreover, the functions $d, D:\tilde {C_T}\to \mathbb R$ are continuous. 
\end{thm}

\begin{proof}
For the reader's convenience, we shall proceed in several steps. 

\medskip

\noindent \underline{Step 1} (An associated integro-differential problem). Observe that if $x\in C^1_T$ is a solution of (\ref{eq1}), then, $\Delta$-integration over $[0,T]_\mathbb T$ yields  $s=\frac{b}{T}\int_0^T \sin (x(t))\Delta t$.  Therefore, it proves convenient to consider the integro-differential Dirichlet problem

\begin{equation}\label{eq1int}
\left\{\begin{array}{ll}
(\varphi(x^{\Delta}(t)))^{\Delta} +ax^{\Delta}(t)+ b\sin x(t) = p_0(t) + s(x),  \ \ \ t\in (0,T)_\mathbb{T}  \\
x(0)=x (T), 
\end{array}\right.
\end{equation}
with  $s(x):=\frac{b}{T}\int_0^T \sin (x(t))\Delta t$. By Schauder's fixed point theorem, 
it is straightforward to prove that for each $r\in \mathbb{R}$ there exists at least one solution $x\in C([0,T]_\mathbb T)$ of (\ref{eq1int}) such that $x(0) = x(T) = r$.

\medskip

\noindent \underline{Step 2} ($\mathcal I(p_{0}$) is is nonempty and bounded). 
Let $x$ be a solution of 
 (\ref{eq1int}) such that $x(0)=x(T)=r$, then integration
 over $[0,T]_\mathbb T$
yields

$$
\varphi(x^{\Delta}(T)) -
\varphi(x^{\Delta}(0)) + b\int_0^T\sin x(t)\Delta t = Ts(x),
$$
and hence
$\varphi(x^{\Delta}(T)) =
\varphi(x^{\Delta}(0))$. It follows that $x$ may be extended in a $T$-periodic fashion to a solution of (\ref{eq1}) with $s=s(x)$. 
In other words,
$$\mathcal I(p_0)= \{ s(x): x \hbox{ is a solution of (\ref{eq1int}) for some } r\in [0,2\pi]\} \ne \emptyset.
$$
Moreover, it is clear from definition that  $\left|s(x)\right|\leq b$, so $\mathcal I(p_{0})\subset [-b,b]$.

\medskip

\noindent \underline{Step 3} ($\mathcal I(p_{0}$)  is connected). Assume that   $s_{1}, s_{2}\in \mathcal I(p_{0})$ are such that $s_{1} < s_{2}$, and let $x_{1}$ and $x_{2}$  be $T$-periodic solutions of (\ref{eq1}) 
for $s_{1}$ and $s_{2}$, respectively. Then for any $s\in (s_{1}, s_{2})$ it is verified that $x_{1}$  and $x_{2}$ are   strict upper and a lower solutions of (\ref{eq1}), respectively. 
Replacing  $x_{1}$  by  $x_{1}+ 2k\pi$, with $k$ the first integer  such that $x_{2}<x_1 + 2k\pi$ and applying  Theorem \ref{te2} with $\alpha =x_2$ and $\beta=x_1 + 2k\pi$, we conclude that  problem $(\ref{eq1})$ has at least one $T$-periodic solution, whence $s\in \mathcal  I(p_{0})$.

\medskip

\noindent \underline{Step 4} ($\mathcal I(p_{0}$) is closed). Let $\left\{s_{n}\right\} \subset \mathcal I(p_{0})$ converge 
to  some $s$, and let $x_{n}\in C^1_T$ be a solution of (\ref{eq1}) for $s_n$. Without loss of generality, we may assume that $x_{n}(0) \in [0,2\pi]$. 
Because $\left\| x^{\Delta}_{n} \right\|_{\infty}<c$, by
 Arzel\`a-Ascoli theorem there exists a subsequence (still denoted  $\left\{x_{n}\right\}$)  that converges uniformly to some $x$.
Furthermore, from (\ref{eq1}) we deduce the existence of a constant $C$ independent of $n$ such that $ |(\varphi(x^{\Delta}_{n}(t)))^{\Delta}| \le C$ for all $t$. We claim that $\varphi(x^\Delta_n)$ is also uniformly bounded, that is, $\|x_n^\Delta\|_\infty$ is bounded away from $c$. Indeed, otherwise passing to a subsequence we may suppose for 
example that 
$\varphi(x^\Delta_n)_{\max} \to +\infty$. Because 
$\varphi(x_n^\Delta(t_1)) - \varphi(x_n^\Delta(t_0)) \le C(t_1-t_0) $
for all $t_1>t_0$, we deduce from periodicity that 
$\varphi(x^\Delta_n)_{\max} - \varphi(x^\Delta_n)_{\min} \le CT$ and, consequently, 
$\varphi(x_n^\Delta)_{\min}\to +\infty$. This implies that $(x_n^\Delta)_{\min}\to c$, which contradicts the fact that $x_n^\Delta$ has zero average. 
Using Arzel\`a-Ascoli again, we may assume that $\varphi(x^\Delta_n)$ converges uniformly to some function $v$ and, from the identity $x_n(t)=x_n(0)+\int_0^t x^\Delta_n(\xi)\Delta\xi$ we deduce that $x\in C^1_T$ and $x^\Delta = \varphi^{-1}(v)$. Now integrate the equation for each $n$ and take limit for $n\to\infty$ to obtain
$$
\varphi(x^\Delta(t)) = 
\varphi(x^\Delta(0)) + \int_0^t [s + p_0(\xi) -b\sin(x(\xi))]\Delta\xi  - a[x(t)-x(0)].
$$
In turn, this implies that 
$x$ is a solution of (\ref{eq1int}) with $s(x) = s$; hence, $\mathcal I(p_{0})$ is closed  and the proof is complete.

\medskip

\noindent \underline{Step 5} (continuous dependence on $p_0$). 
Let $\{p_0^n\}_{n\in \mathbb N}\subset \tilde{C_T}$ be a sequence that converges to some $p_0$. We shall prove that $D(p_0^n)\to D(p_0)$; the proof for $d$ is analogous. Similarly to Step 4, it is seen that if a subsequence of $\{ D(p_0^n)\}$ converges to some $D$, then the problem for $p_0$ with $s=D$ admits a solution and, consequently, $D\le D(p_0)$. Thus, it suffices to prove that $\liminf_{n\to\infty} D(p_0^n) \ge D(p_0)$. 
Indeed,  otherwise, passing to a subsequence 
we may suppose that $D(p_0^n)\to D< D(p_0)$. Fix $\eta >0$ such that $D+\eta <D(p_0)$ 
and let $x$ be a $T$-periodic solution of (\ref{eq1}) for $s=D(p_0)$. Take $n$ large enough such that
$$
p_0(t) + D(p_0) > p_0^n(t) + D + \eta > p_0^n(t) + D(p_0^n) \qquad \forall\, t\in [0,T]_\mathbb T
$$
and let $x_n$ be a $T$-periodic solution of (\ref{eq1}) for $p_0^n$ and $s=D(p_0^n)$. The previous inequalities imply that $x$ and $x_n$ are respectively a lower and an upper solution of the problem for $p_0^n$ and $s=D+\eta$ and, without loss of generality, we may assume that $x<x_n$. Thus, (\ref{eq1}) has a $T$-periodic solution for $p_0^n$ and $s=D+\eta >D(p_0^n)$, a contradiction.

\end{proof}

\medskip



The following theorem establishes the existence of at least two geometrically different $T$-periodic solutions to problem (\ref{eq1}) when $s$ is an interior point. 

\begin{thm}\label{te4}
Assume that  $p_{0}\in C_T$ has zero average. 
If $s\in \left(d(p_{0}), D(p_{0})\right)$, then the problem (\ref{eq1}) has at least two geometrically different $T$-periodic solutions.

\end{thm}

\begin{proof}
For $s\in \left(d(p_{0}), D(p_{0})\right)$, let $s_{1}:=d(p_{0})<s< D(p_{0}):=s_{2}$ and let
$x_1$, $x_2$ be  as in Step 3 of the previous proof. Then
 $x_{1}$ and $x_{2}$ are strict upper  and lower solutions for $s$,  respectively. Due to the $2\pi$-periodicity of (\ref{eq1}), we may assume that 
$x_{2}<x_1$ and $x_2 + 2\pi \not\le x_1$ and, consequently, 
$\Omega_{x_2,x_1}$ and $\Omega_{x_2+2\pi,x_1+2\pi}$ are disjoint open  subsets of 
$\Omega_{x_2,x_1+2\pi}$.
From Theorem \ref{te2} and the excision property of the Leray-Schauder degree, we deduce the existence of three different solutions $y_1, y_2, y_3\in C^1_T$ such that 
\begin{center}
   $x_2(t) <y_1(t)<x_1(t),$ \\ 
   $x_2(t)+2\pi< y_2(t)< x_1(t) + 2\pi$\\
      $x_2(t)<y_3(t)< x_1(t)+ 2\pi$
 \end{center}
 for all $t\in {\mathbb T}$.
If  $y_{2} = y_{1}+2\pi$, then $y_{3}\neq y_{1},  y_{1} + 2\pi$ and the conclusion follows. 

\end{proof}

 \section{Sufficient conditions for $0\in \mathcal  I(p_{0})$}

In this section, we shall obtain  conditions guaranteeing that $0$ belongs to the solvability set. Even in the continuous case, this is not clear when $a\ne 0$ since, as {it is} well known, counter-examples exist for the classical pendulum equation for arbitrary periods. 
In the relativistic case, however, it was proved that $0\in \mathcal I(p_0)$ when $T$ is sufficiently small and counter-examples for large values of $T$ are not yet known. 
Here, as mentioned in the introduction,
we shall 
improve the bounds for $T$ obtained in previous works for $\mathbb T=\mathbb R$. 
The results shall be expressed 
in terms of $k(\mathbb T)$, 
the optimal constant of the  inequality
$$
\|x-\overline x\|_\infty \le k \|x^\Delta\|_\infty,\qquad x\in C_T^1.
$$
 For instance, for arbitrary $\mathbb T$ it is readily seen
that $k(\mathbb T)\le \frac T2$, because $x^\Delta$ has zero average and hence, due to periodicity, 
$$x_{\max} - x_{\min} \le \int_{t_{\min}}^{t_{\max}} [x^\Delta(t)]^+ \Delta t \le \int_0^T [x^\Delta(t)]^+\Delta t = \frac 12 \int_0^T|x^\Delta(t)| \Delta t.
$$

We recall that, in the continuous case, the (optimal) Sobolev inequality $\|x-\overline x\|_\infty \le \sqrt{ \frac T{12} } \|x'\|_2$ implies that
 $k(\mathbb R)\le \frac T{2\sqrt 3}$. 

The main result of this section reads as follows. 

 \begin{thm}\label{bound}
 
 Assume that $ck(\mathbb T)<\pi$ and define the function
 $$
\psi(\delta):={2\delta} \cos (\delta)   + (cT- {2\delta} )\cos \left(  {ck(\mathbb T)} \right).
 $$
If $\psi(\delta)\ge 0$ for some $\delta \in (0,\frac \pi 2)$, then $0\in \mathcal I(p_0)$. Furthermore, if the previous inequality is strict, then $0\in \mathcal I(p_0)^\circ$.  
 \end{thm}

Before proceeding to the proof, it is worth to recall   that, from Theorem 4.1 and Example 5.3 in \cite{adsk}, in order to prove the existence of $T$-periodic solutions for $s=0$ it suffices to verify that the equation
\begin{equation}\label{int}
 \left(\frac{x{^\Delta}(t)}{\sqrt{1-\frac{x{^\Delta}(t)^2}{c^{2}}}}\right)^\Delta  =\lambda [p_0(t) -
 ax^\Delta(t)-  b\sin x(t)]
\end{equation} 
has no $T$-periodic solutions with average $\pm \frac\pi 2$. For example, if  $x\in C^1_T$ is a solution of (\ref{int}) such that $\overline x=\frac\pi 2$, then it follows from the definition of $k(\mathbb T)$ that, for all $t\in \mathbb T$,
$$
\left|x(t)-\frac\pi 2\right|\le ck(\mathbb T). 
$$
In particular, if $ck(\mathbb T)\le \frac \pi 2$,  then $x(t)\in [0,\pi]$ for all $t\in \mathbb T$ and, upon integration of equation (\ref{int}),  we deduce:
$$0 = b\int_0^T \sin(x(t))\Delta t >0. 
$$
The same contradiction is obtained also if $\overline x=-\frac \pi 2$. 
For example,  
 the condition $cT\le \pi$ is sufficient for arbitrary $\mathbb T$ and, in the continuous case, the condition $cT\le \sqrt{3}\pi$ is retrieved. 
 However,  
 the previous bound $ck(\mathbb T)\le \frac \pi 2$ can be improved, as we shall see in the following proof. 
 \medskip 
 
 \begin{proof1}
 From the preceding discussion, it may be assumed that  
 $\frac \pi 2< ck(\mathbb T)<\pi$.
  Suppose that $x$ is a solution of (\ref{int})
  such that   $\overline x =\frac \pi 2$, 
 then 
 $$x(t)\in \left[\frac \pi 2- ck(\mathbb T),\frac {\pi}2 +ck(\mathbb T)\right] 
 \subset 
 \left(-\frac \pi 2,\frac {3\pi}2\right)$$
 for all $t\in \mathbb T$ and hence
 $$
 \sin x(t)\ge -\sin(A)>-1, \qquad\hbox{ where } A = {ck(\mathbb T)} -\frac \pi 2.
 $$
Fix $\delta\in (0,\frac \pi2)$ and consider the set
$$C_\delta=\left\{ t\in [0,T]_\mathbb T: |x(t)-\frac\pi 2| \le \delta\right\}.
$$
Then
\begin{equation}\label{contrad}
\begin{array}{ccl}
    0 & = & \int_0^T\sin(x(t))\Delta t \ge \int_{C_\delta} (\sin(x(t))+\sin (A))\Delta t - T\sin (A)  \\
    {}&{}&{}\\
    {} &  \ge &  \left[\sin(\frac \pi 2 -\delta)+\sin (A)\right] \mathfrak m(C_\delta) - T\sin (A)\\
    {}&{}&{}\\
    {} & = & \cos (\delta) \mathfrak m(C_\delta) - [T- \mathfrak m(C_\delta)]\sin A,
\end{array}
\end{equation}
where $\mathfrak m(C_\delta)$ is the measure of the set $C_\delta$ associated to the $\Delta$-integral,  namely $\mathfrak m (C_\delta)= \int_{C_\delta}\Delta t$. 
Clearly, a contradiction is obtained when the latter term of (\ref{contrad}) is positive. 

 Moreover, notice that if $x(t_0)\le \frac \pi 2$ and $t_1>t_0$ is such that $x(t_1) \ge\frac \pi 2 + \delta$, 
 then
 $$
\delta \le  x(t_1)-x(t_0) =\int_{t_0}^{t_1} x^\Delta(s)\Delta s < c(t_1-t_0). 
 $$
 In the same way, if $t_0<t_1$ are such that 
 $x(t_0)\ge \frac \pi 2$ and $x(t_1)\le \frac \pi 2-\delta$, then $c(t_1-t_0)> \delta$. 
Thus, by periodicity, we deduce that $\mathfrak m(C_\delta)> \frac {2\delta} c$. 
The same conclusions are obtained if $\overline x=-\frac \pi 2$; 
hence, a sufficient condition for the existence of at least one $T$-periodic solution is that, for some $\delta\in (0,\frac \pi 2)$, 
 $$
 \cos (\delta) \frac {2\delta} c \ge \left(T- \frac {2\delta} c\right)\sin A
 $$
or, equivalently, that $\psi(\delta)\ge 0$. 
 Note, furthermore, that if the  inequality is strict, then a contradiction is still obtained as 
 in (\ref{contrad}) if we add a small parameter $s$ to the function $p_0$ in (\ref{int}).

 \end{proof1}

\begin{rem}

It is seen  that $\psi$ reaches its maximum at the unique $\delta^*\in (0,\frac \pi 2)$ such that 
\begin{equation}
    \label{delta}
    \cos(\delta^*) - \delta^* \sin(\delta^*) = \cos \left( ck(\mathbb T)\right).
\end{equation}
Thus, replacing (\ref{delta}) in $\psi$, a somewhat explicit condition on $T$ reads:   
 $$2(\delta^*)^2\sin(\delta^*)   +  cT \cos\left(ck(\mathbb T)\right)\ge 0.
 $$

\end{rem}

An immediate corollary is the following:

\begin{cor}\label{coro}
 There exists a constant $T^*$ with $cT^* > \pi$ such that 
 $0\in \mathcal I(p_0) $ for all $p_0\in \tilde C_T$ 
 if $T\le T^*$ and it is an interior point if $T<T^*$. 
 For the particular case $\mathbb T=\mathbb R$,  it is verified that
 $cT^*> \sqrt {3}\pi$.
\end{cor}
 
 \begin{proof}
  For arbitrary $\mathbb T$, we know already that 
   $k(\mathbb T)\le \frac T2$, then 
 a sufficient condition when $cT \in (\pi,2\pi)$ is the existence of $\delta\in (0,\frac \pi 2)$ such that $\Psi(\delta,T)\ge 0$, where
 $$\Psi(\delta,T):={2\delta} \cos (\delta)   + (cT- {2\delta} )\cos \left(  \frac {cT}2 \right).
 $$
The result now follows trivially from the fact that $\Psi(\delta,\frac \pi c)=2\delta\cos (\delta)$. The proof is similar for $\mathbb T=\mathbb R$, now taking
$$\Psi_{cont}(\delta,T):={2\delta} \cos (\delta)   + (cT- {2\delta} )\cos \left(  \frac {cT}{2\sqrt 3} \right).
$$
   
 \end{proof}
  
  \begin{rem}
  A more quantitative version of the previous corollary follows from the 
 fact that the function $\Psi$ is strictly decreasing with respect to $T$ when $cT\in (\pi,2\pi)$ and arbitrary $\delta\in (0,\frac \pi 2)$.
 In particular, observe that if $\Psi (\delta,\hat T)\ge 0$ for some $\hat T\in (\frac\pi c,\frac {2\pi}c)$ and some $\delta\in (0,\frac\pi 2)$,  then $\Psi(\delta,T)>0$ for $T\in (\frac \pi c,\hat T)$. 
 Thus, a lower bound for $T^*$ is given by the unique value of $T\in (\frac \pi c,\frac {2\pi}c)$ such that 
 $$
 \max_{\delta\in [0,\frac \pi 2]} \Psi(\delta,T)=0.
 $$
Analogous conclusions are obtained when $\mathbb T=\mathbb R$
 using $\Psi_{cont}$ instead of $\Psi$. 
  \end{rem}

 
\subsection{Numerical examples and final remarks}

As shown in Corollary \ref{coro}, 
the bound thus obtained always improves the simpler one  $ck(\mathbb T)\le \frac \pi 2$ and, in particular,  it guarantees that if the latter inequality is satisfied 
then
 $0$ is in fact an interior point of $ \mathcal I(p_0)$. 
In the continuous case, 
an easy numerical computation 
gives the sufficient condition
$cT\le 6.318$, slightly better than 
the bound $cT< \sqrt 3 \pi$ obtained
in \cite{ma1} (see Figure 1). 
 For arbitrary $\mathbb T$,   numerical experiments show that $0\in \mathcal I(p_0)^\circ$ for $cT\le 4.19$, as shown in  Figure 2.

\begin{figure}[ht]
\begin{center}
\includegraphics[width=8cm,height=5cm]{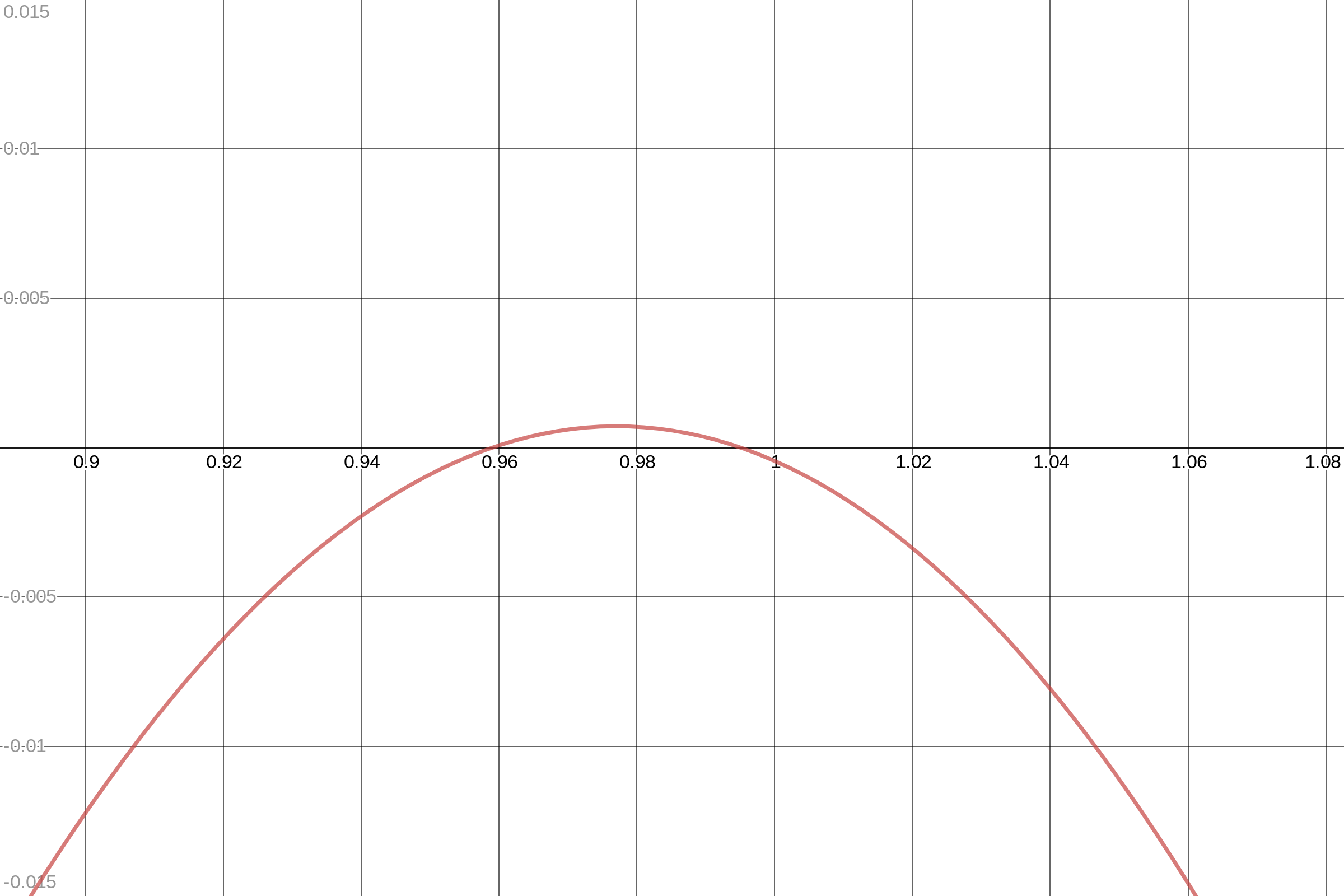}
\caption{Graph of $\psi$ for $\mathbb T=\mathbb R$ with $cT=6.318$}
\end{center}
\end{figure}

\begin{figure}[ht]
\begin{center}
\includegraphics[width=8cm,height=5cm]{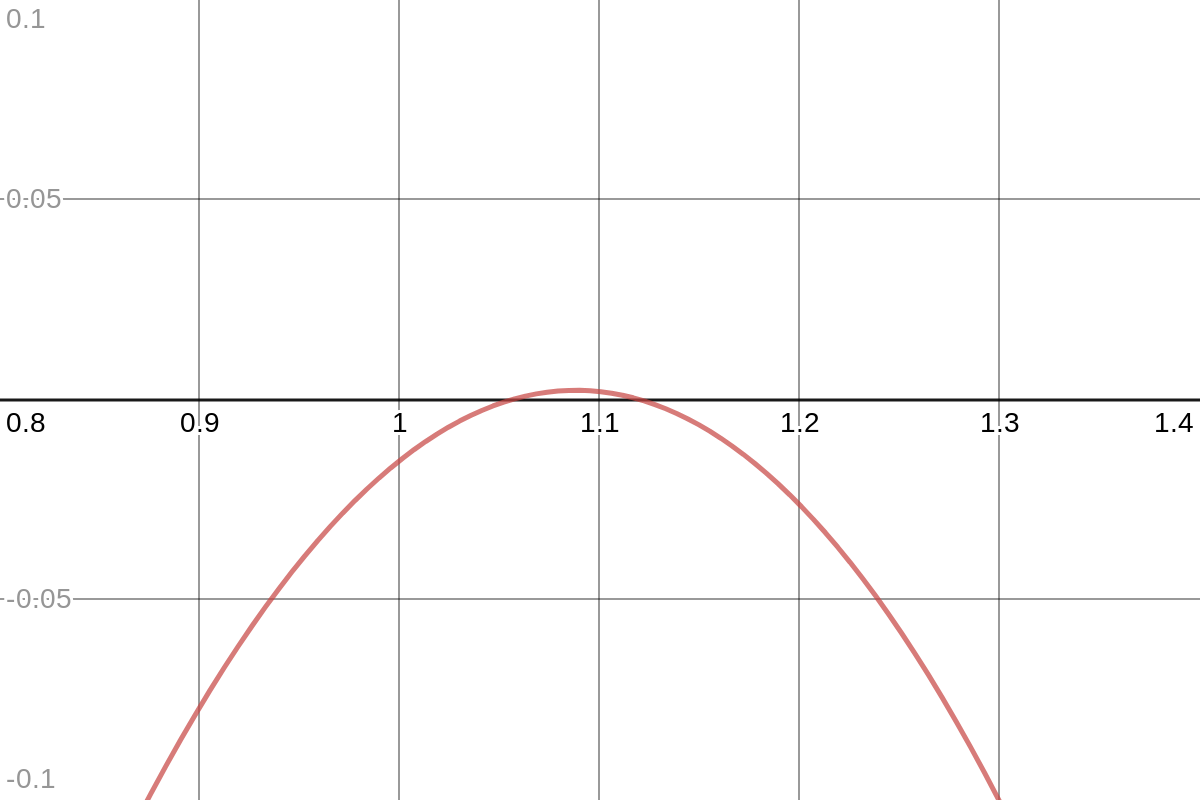}
\caption{Graph of $\psi$ for $k(\mathbb T)=\frac T2$ and $cT=4.19$}
\end{center}
\end{figure}    


\begin{rem}
An estimation of the constant $k(\mathbb T)$ could be obtained analogously to the continuous case as shown for example in \cite{maw}. 
Let $\{e_n\}_{n\in \mathbb Z}\subset C_T$ be an orthonormal basis of $L^2(0,T)_\mathbb T$ with $e_0\equiv \frac 1{\sqrt {T}}$ and $E_n$ be a primitive of $e_n$ such that $\overline E_n=0$. Writing $x^\Delta = \sum_{n\ne 0} a_n e_n$, it follows that 
$$\|x-\overline x\|_\infty = \left|\sum_{n\ne 0} a_n E_n\right| \le \|x^\Delta\|_{L^2} 
\sqrt{\sum_{n\ne 0} \|E_n\|_\infty^2}\le 
 \|x^\Delta\|_\infty 
\sqrt{T\sum_{n\ne 0} \|E_n\|_\infty^2}.
 $$
When $\mathbb T=\mathbb R$, taking the usual Fourier basis one has that $\|E_n\|_\infty = \frac {\sqrt {T}}{2\pi n}$ and the value $k(\mathbb R)\le \frac T{2\sqrt{3}}$ is obtained from the well known equality
$\sum_{n\in\mathbb N} \frac 1{n^2} =\frac {\pi^2}6$. 

\end{rem}

\begin{rem}

As mentioned in the introduction, Theorem \ref{bound}  allows to compute an inferior bound for the length of the solvability interval which
does not depend on $p_0$, provided that $T$ is small enough. In some obvious cases, inferior bounds are obtained for arbitrary $T$: for example, if   $\|p_0\|_\infty < b$ then 
$[-\varepsilon,\varepsilon]\subset \mathcal I(p_0)$
for $\varepsilon = b-\|p_0\|_\infty$.
This is  readily verified taking $\alpha=\frac \pi 2$ and $\beta = \frac {3\pi}2$ as lower and upper solutions. 
\end{rem}


\section*{Acknowledgements} 
This research was partially supported by projects PIP  {11220130100006CO CONICET}
and UBACyT  {20020160100002BA}.

\bibliographystyle{plain}

\begin{thebibliography}{99}

\bibitem{opuscula} P. Amster, M. P. Kuna and D. P. Santos {\em Multiple solutions of Boundary Value Problems on Time Scales for a $\varphi$-Laplacian Operator}, to appear in Opuscula Mathematica.

\bibitem{adsk} P. Amster, M. P. Kuna and D. P. Santos.
{\em Existence  and Multiplicity of Periodic Solutions  for  Dynamic  Equations with  Delay  and singular $\varphi$-laplacian of Relativistic Type}, 
submitted. 


\bibitem{bjm}
C. Bereanu, P. Jebelean and J. Mawhin,
\textit{Periodic Solutions of Pendulum-Like Perturbations of Singular and Bounded $\phi$-Laplacians.}
J Dyn Diff Equat. {22}  (2010), 463--471.

\bibitem{ma1}C. Bereanu and J. Mawhin, \textit{Existence and multiplicity results for some nonlinear problems with singular $\varphi$-laplacian}, J. Differential Equations. 243 (2007), 536--557.

\bibitem{bt} C. Bereanu and P. J. Torres,\textit{Existence of at least two periodic solutions of the forced relativistic pendulum}, Proceedings of the American Mathematical Society (2012): 2713--2719.


\bibitem{bo1} M. Bohner and A. Peterson, \textit{Dynamic Equations on Time Scales}, Birkhauser Boston, Massachusetts, 2001.

\bibitem{bo2} M. Bohner and A. Peterson (eds.), \textit{Advances in Dynamic Equations on Time Scales}, Birkhauser Boston, Massachusetts, 2003.

\bibitem{bre}H. Brezis, J. Mawhin, \textit{Periodic solutions of the forced relativistic pendulum}, Differential
Integral Equations. 23(9) (2010) 801--810.

\bibitem{castro} 
 A. Castro, \textit{Periodic solutions of the forced pendulum equation}. Diff. Equations (1980), 149--160.



\bibitem{fm} {G. Fournier and J. Mawhin, \textit{On periodic solutions of forced pendulum-like equations}, J. Differential Equations 60.3 (1985): 381--395.}



\bibitem{hilgerthe}S. Hilger, \textit{Ein Ma$\beta$kettenkalk\"ul mit Anwendung auf Zentrumsmanningfaltingkeiten}, PhD thesis, Universit\"at W\"urzburg, 1988.



 \bibitem{maw} J. Mawhin,
 \textit{Degr\'e topologique
 et solutions
 p\'eriodiques
 des syst\`emes
 diff\'erentiels
non lin\'eaires},
 Bull. Sot. Roy. Sci. Li\`ege 38 (1969),
 308--398.
 
 
\bibitem{man11} J. Mawhin, \textit{Periodic solutions of the forced pendulum: classical vs relativistic}, Matematiche (Catania). 65(2) (2010), 97--107.

 \bibitem{maw-75} J. Mawhin, \textit{Seventy-five years of global analysis around the forced pendulum equation}. In: Proceedings of Equadiff 9, Masaryk University, Brno (1997), 115--145.  


\bibitem{man12}J. Mawhin, \textit{Topological Degree Methods in Nonlinear Boundary Value Problems}, CBMS series No. 40, American Math. Soc., Providence RI, 1979.


\bibitem{pras}
K. Prasad and P. Murali, 
\textit{Heteroclinic Solutions of Singular $\phi$-Laplacian
Boundary Value Problems on Infinite Time Scales}.
 Electronic Journal of Qualitative Theory of Differential Equations 2012 (2012), 1--9. 





\bibitem{torres} P. Torres, \textit{Periodic oscillations of the relativistic pendulum with friction}, Phys. Lett. A.
372(42) (2008) 6386--6387.

\bibitem{torres2} P. Torres
\textit{Nondegeneracy of the periodically forced Li\'enard differential equation with $\phi$-laplacian. }
Communications in Contemporary Mathematics
13, No. 2 (2011) 283--292.




\end{thebibliography}

\end{document}